\documentclass[reqno,tbtags,11pt]{amsart}

\usepackage{microtype}
\usepackage{mathtools}
\usepackage{siunitx}
\usepackage{graphicx}
\usepackage{csquotes}
\usepackage{tabularray}
\UseTblrLibrary{booktabs}

\newcommand{\dd}{\mathop{}\!{\mathrm {d}}}
\newcommand{\ii}{\mathrm {i}}
\newcommand{\ee}{\mathrm {e}}

\DeclarePairedDelimiter{\abs}         {\lvert}{\rvert}
\DeclarePairedDelimiter{\norm}        {\lVert}{\rVert}
\DeclarePairedDelimiter{\innerproduct}{\langle}{\rangle}

\newcommand{\Quadraticform}{\mathfrak {h}}
\newcommand{\I}{\mathfrak {i}}
\newcommand{\Ab}{\mathbf {A}}

\newtheorem{theorem}{Theorem}
\newtheorem{lemma}[theorem]{Lemma}
\numberwithin{equation}{section}

\author[C. Léna]{Corentin Léna}
\address
  [C. Léna]
  {University of Padua, Department of Management and Engineering - DTG, Stradella
  S. Nicola 3, 36100 Vicenza and Department of Mathematics \enquote {Tullio
  Levi-Civita}, via Trieste 63, 35121 Padua, Italy.}
\email{corentin.lena@unipd.it}

\author[M. Sundqvist]{Mikael Sundqvist}
\address
  [M. Sundqvist]
  {Department of Mathematics, Lund University, Sweden}
\email{mikael.persson\_sundqvist@math.lth.se}

\title
  {A magnetic eigenvalue bound in the disk}

\begin{document}

\begin{abstract}
  We consider the magnetic Schrödinger operator in the unit disk with constant
  magnetic field of strength \(b>0\) and magnetic Neumann boundary condition. If
  \(\lambda_1(b)\) denotes its lowest eigenvalue, then we prove that
  \(\lambda_1(b) < \Theta_0 b\) for all \(b>0\), where \(\Theta_0\) is the de
  Gennes constant. The proof has two parts, both based on Rayleigh's principle.
  For large \(b\), we use a trial state built from the de Gennes ground state.
  For the remaining bounded range of \(b\), we divide the interval into finitely
  many overlapping subintervals and, on each of them, choose a trial state from a
  finite-dimensional space. This reduces the proof to finitely many inequalities
  between rational numbers.
\end{abstract}

\maketitle

\section{Introduction}

We consider the magnetic Schrödinger operator in the unit disk \(\mathbb D\),
with constant magnetic field of strength \(b>0\) and magnetic Neumann boundary
condition. More explicitly, we choose a magnetic vector potential \(\Ab =
(b/2)(-x _ 2, x _ 1)\) (known as a rotationally invariant gauge) and define the
quadratic form
\[
  \Quadraticform[\psi]
  =
  \int_{\mathbb D} \abs{(\ii\nabla+\Ab)\psi}^2 \dd x
\]
for $\psi\in H^1(\mathbb D)$. According to standard results in spectral theory
(see for instance \cite[Chapter 1]{MR2662319}), there exists a unique
self-adjoint operator $\mathcal H(b)$ with a domain $\mathcal D$ dense in
$H^1(\mathbb D)$ and satisfying $\Quadraticform[{\psi,\varphi}]=\innerproduct
{\mathcal H(b)\psi,\varphi}$ for all $\psi\in\mathcal D$ and $\varphi\in
H^1(\mathbb D)$. This is, by definition, the magnetic Schrödinger operator
considered in this note. Since $H^1(\mathbb D)$ is compactly embedded in
$L^2(\mathbb D)$, $\mathcal H(b)$ has discrete spectrum; we denote by
\(\lambda_1(b)\) its lowest eigenvalue. According to Rayleigh's principle (a
special case of the Courant--Fisher min-max formula),
\[
  \lambda_1(b)=\inf_{\psi\in H^1(\mathbb D)}\frac{\Quadraticform[\psi]}{\norm {\psi}^2}.
\]
In particular, for any $b>0$ and any $\psi_0\in H^1(\mathbb D)$, called a
\emph{trial state}, we have
\[
  \lambda_1(b)\le \frac{\Quadraticform[\psi_0]}{\norm {\psi_0}^2},
\]
where the right-hand side is called the Rayleigh quotient for $\psi_0$. We use this
fact repeatedly in the present note.

For a more general two-dimensional smooth bounded domain $\Omega$, the function
$b\mapsto \lambda_1(\Omega,b)$, defined as above, plays an important role in the
Ginzburg-Landau theory of superconductivity. This motivated the development of an
extensive asymptotic theory when $b\to+\infty$ (see \cite{MR2662319} and
references therein). It was found that
\begin{equation}\label{eq:asymptotics}
  \lambda_1(\Omega,b)
  =
  \Theta_0 b-C_1\kappa(\partial\Omega)b^{1/2}
  +
  O(b^{1/3}),
\end{equation}
where $\kappa(\partial\Omega)$ is the maximum of the curvature of
$\partial\Omega$, and $\Theta_0$ and $C_1$ are universal constants (see
\cite[Theorem 8.3.2]{MR2662319}). The \emph{de Gennes constant} \(\Theta_0\) is
the minimum of the lowest eigenvalue of the Neumann harmonic oscillator on the
half-line and $C_1=\varphi(0)^2/3$, with $\varphi$ the corresponding normalized
ground state. We recall the relevant facts about this one-dimensional eigenvalue
problem in Appendix~\ref{sec:deGennes}.

For an arbitrary smooth domain $\Omega$, it follows from \eqref{eq:asymptotics}
that $\lambda_1(\Omega,b)<\Theta_0 b$ for $b$ large enough. It has been proved
in~\cite[Theorem 2.2]{MR4659291} that the inequality holds for all $b>0$ when
$\Omega$ belongs to some particular classes of domains with corners. The same
holds true for infinite sectors if the opening angle is sufficiently small or
close to \(\pi\), see~\cite{MR3906610,MR4780715} and the references therein. On
the basis of numerical and asymptotic evidence, it has been conjectured that
the inequality also holds in the unit disk (see~\cite[Conjecture 1.3]{MR4947381}
and Figure~\ref{fig:eigenvalue}). The goal of this note is to prove that result,
summarized in the following theorem.

\begin{theorem}\label{thm:main}
  For all \( b > 0 \), it holds that \( \lambda _ 1 (b) < \Theta _ 0 b\).
\end{theorem}

Our proof is divided into two parts, corresponding respectively to \enquote
{small} and \enquote {large} \(b\). We first build a trial state from the de
Gennes ground state $\varphi$ and use it to prove that the inequality holds for
all $b$ above an explicit threshold, depending on the constants $\Theta_0$ and
$C_1$. For the remaining \(b\) below this threshold, we use trial states
picked from a finite-dimensional space of polynomial functions and verify the
required inequalities on finitely many overlapping intervals by exact rational
computations.

\begin{figure}[htb]
  \includegraphics[width=\textwidth]{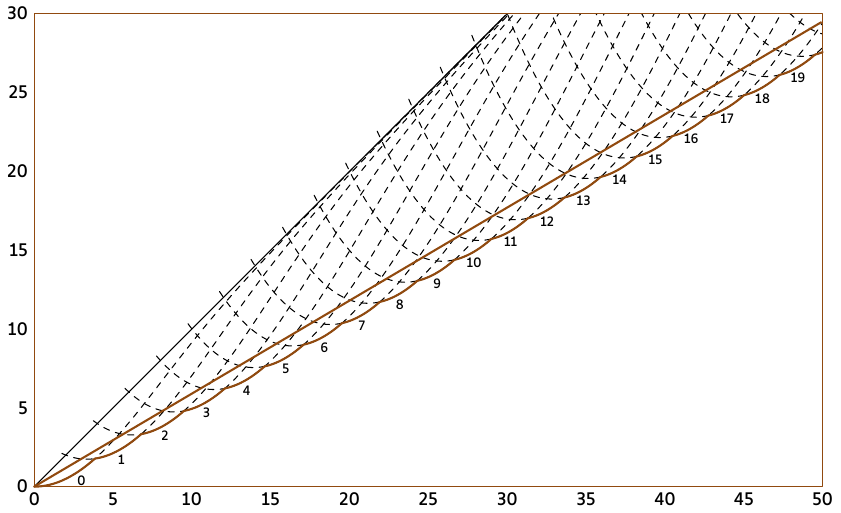}
  \caption{A reproduction of the figure by Saint-James \cite {SAINTJAMES196513}.
  We calculated the eigenvalue curves for a few \(m\) with the help of Mathematica.
  The smallest eigenvalue \(\lambda _ 1(b)\) and the line \(b \mapsto \Theta _ 0
  b\) are drawn with thicker lines.}
  \label{fig:eigenvalue}
\end{figure}

\section{Large intensities}

\begin{lemma}
  If \( b > 130 \) then \(\lambda _ 1 (b) < \Theta _ 0 b\).
\end{lemma}

\begin{proof}
Let \(b > 0\) be arbitrary, and let \( \varphi \) be the de Gennes function, \(C
_ 1 = \varphi(0) ^ 2/3\), and let \( \xi _ 0 > 0 \) be the spectral constant from
Appendix~\ref{sec:deGennes}. We will work with the trial state given, in polar
coordinates \((r,\theta)\), by
\[
  u(r,\theta) = \varphi(\sqrt {b} \ln (1/r))\frac {\ee ^ {\ii m\theta}}{\sqrt {2\pi}}.
\]
In physical terms, the integer $m$ can be interpreted as the \emph{angular
momentum} of the trial state. In practice, for a given $b>0$, we have to choose
$m$ so that the Rayleigh quotient is small enough.

With the change of variable \(t = \sqrt {b} \ln (1/r)\), we find that
the quadratic form can be written as
\[
  \begin{aligned}
    \Quadraticform[u]
    &
    = \int _ 0 ^ {+\infty}
    \Bigl[
      \sqrt {b}\varphi'(t) ^ 2
      + \frac{1}{\sqrt{b}}(m - b \ee ^ {-2t/\sqrt{b}}/2) ^ 2 \varphi(t) ^ 2
    \Bigr] \dd t
    \\
    &
    =
    \frac {\sqrt {b}\xi _ 0 ^ 2}{2}
    + \frac {m ^ 2}{\sqrt {b}}
    + \int _ 0 ^ {+\infty}
      \Bigl[
        \frac {b ^ {3/2}}{4} \ee ^ {-4t/\sqrt {b}} \varphi(t) ^ 2
        - m\sqrt {b} \ee ^ {-2t/\sqrt {b}} \varphi(t) ^ 2
      \Bigr] \dd t
  \end{aligned}
\]
and the norm becomes
\[
  \norm {u} ^ 2
  =
  \int _ 0 ^ {+\infty} \frac {1}{\sqrt {b}} \ee ^ {-2t/\sqrt {b}} \varphi(t) ^ 2 \dd t.
\]

For our purpose, it is more convenient to consider a slightly different quadratic
form. Let us define \(\I[u] = \Quadraticform[u] - \Theta _ 0 b \norm {u} ^ 2\),
where \(\Theta _ 0 = \xi _ 0 ^ 2\). According to Rayleigh's principle, if, for a
given $b>0$, we can find an integer $m$ such that $\I[u]<0$, then
$\lambda_1(b)<\Theta_0 b$.

From Maclaurin expansions of the exponential function, we find that, for \( x > 0
\), the odd-degree Taylor polynomials lie below \(\ee ^ {-x}\), while the
even-degree ones lie above it. In particular
\[
  1 - x  + \frac {x ^ 2}{2} - \frac {x ^ 3}{6}
  <
  \ee ^ {-x}
  <
  1 - x  + \frac {x ^ 2}{2} - \frac {x ^ 3}{6} + \frac {x ^ 4}{24},
  \quad x > 0.
\]
We therefore find, with the momentum integrals \( T _ k \coloneq \int _ 0 ^
{+\infty} t ^ k \varphi(t) ^ 2 \dd t\), that
\[
  \begin{aligned}
    \I[u]
    &
    \leq
    \frac {b ^ {1/2}\xi _ 0 ^ 2}{2}
    + m ^ 2 b ^ {-1/2}
    + \frac {1}{4} b ^ {3/2}
    - b T _ 1
    + 2 b ^ {1/2} T _ 2
    - \frac {8}{3} T _ 3
    + \frac {8}{3} b ^ {-1/2} T _ 4
    \\
    & \qquad
    - m b ^ {1/2}
    + 2 m T _ 1
    - 2 m b ^ {-1/2} T _ 2
    + \frac {4}{3} m b ^ {-1} T _ 3
    \\
    & \qquad
    - \xi _ 0 ^ 2 b
    \bigl(
      b ^ {-1/2}
      - 2 b ^ {-1} T _ 1
      + 2 b ^ {-3/2} T _ 2
      - \frac {4}{3} b ^ {-2} T _ 3
    \bigr).
  \end{aligned}
\]
The right-hand side is a quadratic polynomial in the angular momentum \( m \)
which has its minimum for \( m = m _ {\mathrm {opt}}\), where
\[
  m _ {\mathrm {opt}}
  =
  \frac {b}{2} - T _ 1 b ^ {1/2} + T _ 2 - \frac {2}{3} T _ 3 b ^ {-1/2}.
\]
Since \( m \) has to be an integer, we cannot always use this value, but we can
choose \( m = m _ {\mathrm {opt}} + \epsilon \), with \( \abs {\epsilon} \leq
1/2\). Substituting this value, we obtain one term involving \(\epsilon ^ 2\),
with positive coefficient because the coefficient of \(m ^ 2\) is positive. We
therefore use the estimate \(\epsilon ^ 2 \leq 1/4\); the remaining terms are
independent of \(\epsilon\). Collecting powers of \(b\), we obtain
\[
  \begin{aligned}
    \I[u]
    & \leq
    \Bigl(
      T _ 2
      - T _ 1 ^ 2
      - \frac {\xi _ 0 ^ 2}{2}
    \Bigr) b ^ {1/2}
    +
    \bigl(
      2T _ 1 T _ 2
      - 2T _ 3
      + 2 \xi _ 0 ^ 2 T _ 1
    \bigr)
    \\
    &\qquad
    +
    \Bigl(
      \frac {8}{3} T _ 4
      -\frac {4}{3} T _ 1 T _3
      - T _ 2 ^ 2
      - 2 \xi _ 0 ^ 2 T _ 2
      + \frac {1}{4}
    \Bigr) b ^ {-1/2}
    \\
    &\qquad
    +
    \Bigl(
      \frac {4}{3} T _ 2 T _ 3
      + \frac {4}{3} \xi _ 0 ^ 2 T _ 3
    \Bigr) b ^ {-1}
    -
    \frac {4}{9} T _ 3 ^ 2 b ^ {-3/2}.
  \end{aligned}
\]
We insert the values of the first few momentum integrals (see
Appendix~\ref{sec:deGennes}),
\[
  T _ 1 = \xi _ 0,\quad
  T _ 2 = \frac {3}{2}\xi _ 0 ^ 2,\quad
  T _ 3 = \frac {C _ 1}{2} + \frac {5}{2} \xi _ 0 ^3,\quad
  T _ 4 = \frac {3}{8} + \frac {35}{8}\xi _ 0 ^ 4 + \frac {7}{8} C _ 1 \xi _ 0.
\]
Substituting these values into \(\I[u]\), and throwing away the negative \(b ^
{-3/2}\) term, we find that the inequality simplifies to
\begin{equation}\label{eq:secondtermbound}
  \I[u]
  \leq
  - C _ 1
  + \Bigl(\frac {5}{4} + \frac {5}{3} C _ 1 \xi _ 0
  + \frac {37}{12} \xi _ 0 ^ 4\Bigr) b ^ {-1/2}
  + \frac {5}{3} \xi _ 0 ^ 2(C _ 1 + 5 \xi _ 0 ^ 3)b ^ {-1}.
\end{equation}
We can write this inequality as \(\I[u] \leq -A + B b ^ {-1/2} + C b ^ {-1}\)
where \(A\), \(B\) and \(C\) are positive constants. Then it holds that
\(\I[u] < 0\) if \( b > b _ 0 \), with
\begin{equation}\label{eq:b0}
  b _ 0 = \frac {\bigl(B + \sqrt {B ^ 2 + 4AC}\bigr) ^ 2} {4A ^ 2}.
\end{equation}
Inserting the values of the spectral constants \(\xi_0\) and \(C_1\), we find
that \( b _ 0 \approx 127.4 \), rounded to one decimal place. We now derive a
rigorous bound. It was proved in~\cite[Theorem 1.1]{MR2912745} that
\[
  \abs {\Theta _ 0 -  \num {0.590106125}} \leq 10 ^ {-9},\quad
  \abs {\varphi(0) - \num {0.8730}} \leq 10 ^ {-4}.
\]
We use the following bounds that follow from the above estimates (recall
that \( \xi _ 0 = \sqrt {\Theta _ 0}\) and \(C _ 1 = \varphi(0) ^ 2/3\)):
\[
  \frac {\num {5901}}{\num {10000}} < \Theta _ 0 < \frac {5902}{\num {10000}},
  \quad
  \frac {\num {7681}}{\num {10000}} < \xi _ 0 < \frac {\num {7682}}{\num {10000}},
  \quad
  \frac {253}{\num {1000}} < C _ 1 < \frac {255}{\num {1000}}.
\]
For the constants \(A\), \(B\) and \(C\) above we get
\[
    \frac {253}{\num {1000}} < A < \frac {255}{\num {1000}},
    \quad
    B < \frac {\num {795 156 337}}{\num {300 000 000}},
    \quad
    C < \frac {\num {37 211 493 241}}{\num {15 000 000 000}}.
\]
Inserting these bounds, we find that
\[
  B ^ 2 + 4AC
  <
  \frac {\num {860 007 938 906 177 569}}{\num {90 000 000 000 000 000}}
  <
  \frac {\num {864 900 000 000 000 000}}{\num {90 000 000 000 000 000}}
  =
  \Bigl(\frac {31}{10}\Bigr) ^ 2.
\]
Finally, by substituting these bounds in~\eqref{eq:b0}, we find that
\[
  b _ 0 < {\frac {\num {2 976 164 387 091 257 569}}{\num {23 043 240 000 000 000}}
        <  \frac {\num {2 995 621 200 000 000 000}}{\num {23 043 240 000 000 000}}
        = 130.}
\]

It follows that at least for \(b > 130\) we have \( \lambda _ 1(b) < \Theta _ 0 b \).
\end{proof}

\section{Small intensities}

\begin{lemma}
  If \( 0 < b \leq 131 \) then \(\lambda _ 1 (b) < \Theta _ 0 b\).
\end{lemma}

\begin{proof}
First, by inserting a constant function as trial state, we find that
\( \lambda _ 1 (b) \leq b ^ 2/8\), and so \(\lambda _ 1(b) < \Theta _ 0 b\)
for \( 0 < b < 8\Theta _ 0 \). Since \(\Theta _ 0 > 1/2\) we only need
to prove the inequality for \(b \geq 4\).

For a fixed angular momentum \(m\) we will use a trial state of the form
\begin{equation}\label{eq:trialstate}
  u(r,\theta) = \sum _ {j = 0} ^ N c _ j \psi _ j(r)\frac {\ee ^ {\ii m\theta}}{\sqrt {2\pi}},
  \quad \text {where} \quad
  \psi _ j (r) = r ^ m (1 - r ^ 2) ^ j
\end{equation}
and where we will take \(N = 8\).

We will show that for \(3 \leq b \leq 131\) it is possible to choose \( m \) and
the coefficients \(c _ j\) so that \(\Quadraticform[u] < \Theta _ * b \norm {u}
^2\), where \(\Theta _ * = 5901/10000 < \Theta _0 \). As in the previous section,
this establishes the desired inequality for the corresponding value of $b$. We
actually proceed in the reverse direction, by considering increasing values of
$m$ and constructing for each an interval of values for $b$ in which the
inequality holds, until we have covered the interval $3\le b\le 131$.

We first carried out the computation numerically, and then verified it in
Mathematica using only integers and rational numbers. In particular, no
rounding errors occur in the verification step.

When calculating the energy of our trial state we will need the beta function,
\[
  B(z,w)
  = \int _ 0 ^ 1 t ^ {z - 1}(1 - t) ^ {w - 1} \dd t
  = \frac {\Gamma(z)\Gamma(w)}{\Gamma(z + w)}.
\]
It is defined for \(z\) and \(w\) with positive real part, and it is rational
when \(z\) and \(w\) are integers.

Let us introduce the matrices \(M\), \(K\) and \(P\) that we need for our
calculations. To calculate the integrals below, we use the change of variable \(t
= r ^ 2\) and then identify the result as beta functions. First we have the mass
matrix
\[
  M _ {j,k} = \innerproduct {\psi _ j(r),\psi _ k(r)} = \frac {1}{2} B(m + 1,j + k + 1).
\]
Then we have the kinetic part, as well as the part of the potential that is
independent of \(b\),
\[
  \begin{aligned}
    K _ {j,k}
    & = \innerproduct {\psi _ j ' (r),\psi _ k '(r)}
      + \innerproduct {(m^2/r^2) \psi _ j(r),\psi _ k(r)}\\
    & = \frac {1}{2}
      \bigl[
        2m^2 B(m,j + k + 1)
        \\
        &\qquad\qquad
        - 2m(j + k)B(m + 1,j + k)
        + 4jk B(m + 2,j + k -1)
      \bigr].
  \end{aligned}
\]
Here we interpret the terms with non-positive second argument in the
function as zero. Finally, we have the potential part that stands in
front of \(b ^ 2\) in the form,
\[
  P _ {j,k}
  =
  \innerproduct {(r^2/4) \psi _ j(r),\psi _ k(r)}
  =
  \frac {1}{8}B(m + 2,j + k + 1).
\]
We stress that all the entries in the matrices are rational numbers.

With these notations and with any \(c = (c _ 0,\ldots,c _ {N})\) constructed from
coefficients \(c _ j\) in~\eqref{eq:trialstate} we get a trial state \(u\), and
\[
  \norm {u} ^ 2 = \innerproduct{c,Mc},\quad
  \Quadraticform[u] = \innerproduct{c,Kc}
  - mb\innerproduct{c,Mc}
  + b ^ 2 \innerproduct{c,Pc}.
\]

To obtain a first approximate \(c\) we proceed as follows. For each \(m\) we set
the initial \(b\) to \(b _ {\mathrm {ini}} \coloneq 2m + 2.25\sqrt{m}\). Then we
calculate, numerically, the smallest eigenvalue of the matrix \(K - mb _ {\mathrm
{ini}} M + b _ {\mathrm {ini}} ^ 2 P\) with respect to the mass matrix \( M\). We
take as \(c\) the corresponding eigenvector, normalized to be \(1\) in the first
entry. Then we build the quadratic polynomial \(p\),
\[
  p(b)
  =
  \innerproduct{c,Kc}
  - (m + \Theta _ *) b\innerproduct{c,Mc}
  + b ^ 2 \innerproduct{c,Pc}.
\]
Assuming that $p$ has two real zeros \(b _ {\mathrm {min}} < b _ {\mathrm
{max}}\), we have
\[
  \Quadraticform[u] - \Theta _ 0 b \norm {u} ^2<p(b)
  =
  \Quadraticform[u] - \Theta _ * b \norm {u} ^2<0
\]
for all \(b\in(b _ {\mathrm {min}},b _ {\mathrm {max}})\). Here we have to choose
\(N\), the dimension of the space of trial functions, large enough so that \( p
\) has two real zeros. It turned out that \(N = 8\) was sufficient for our
needs.

From these numerical calculations we now build rational, in fact integer,
endpoints and rational coefficients, so that the final verification uses only
rational numbers.

We first set \(b _ {\mathrm {left}} = \lceil b _ {\mathrm {min}} \rceil\) and \(b
_ {\mathrm {right}} = \lfloor b _ {\mathrm {max}}\rfloor\). It turned out that we
always got \(b _ {\mathrm {left}} < b _ {\mathrm {right}}\), see
Table~\ref{tab:intervaltable}.

Next, we rationalize the \(c\) vector in Mathematica. The only purpose of
that step is to work with rational values. We use a granularity of \(10 ^
{-2}\). We call the rational vector \(c _ {\mathrm {rat}}\), and build a trial
state \(u _ {\mathrm {rat}}\) from it. Then we insert this trial state with the
integers \(b = b _ {\mathrm {left}}\) and \(b = b _ {\mathrm {right}}\) and
verify that \emph {the rational numbers}
\[
  \Quadraticform[u _ {\mathrm {rat}}]
  - \Theta _ * b \norm {u _ {\mathrm {rat}}} ^2
  =
  \innerproduct{c _ {\mathrm {rat}},Kc _ {\mathrm {rat}}}
  - (m + \Theta _ *) b\innerproduct{c _ {\mathrm {rat}},Mc _ {\mathrm {rat}}}
  + b ^ 2 \innerproduct{c _ {\mathrm {rat}},Pc _ {\mathrm {rat}}}
\]
are negative. Since the parabola is convex and since \(\Theta _ * < \Theta _ 0\)
this proves that \(\lambda _ 1(b) < \Theta _ 0 b\) for \(b \in [b _ {\mathrm
{left}},b _ {\mathrm {right}}]\). In Table~\ref{tab:intervaltable} we present the
overlapping intervals we get for the different \(m\) that we use. The above
expression is indeed negative at both endpoints of each of those intervals. In
Appendix~\ref{sec:details} we give the full Mathematica code and the output we
get from it.

To sum up, the constant trial state gives the claimed inequality for
\(0 < b \leq 4\). For \(3 \leq b \leq 131\), Table~\ref{tab:intervaltable}
provides overlapping intervals on which the inequality is verified by explicit
trial states and exact rational computations. Hence
\(\lambda _ 1(b) < \Theta _ 0 b\) for all \(0 < b \leq 131\).
\end{proof}

\section*{Acknowledgements}

C.~L. acknowledges support from the INdAM GNAMPA Project \emph{Functional and
spectral analysis for differential operators} (CUP E53C25002010001) and would
like to thank the Isaac Newton Institute for Mathematical Sciences, Cambridge,
for support and hospitality during the programme \emph{Geometric spectral theory
and applications} (EPSRC grant EP/Z000580/1), where part of this work was carried
out.

\appendix
\section{The de Gennes model}\label{sec:deGennes}

The parameter-dependent family of operators \(\mathcal G(\xi) = -d ^ 2/d t ^ 2 +
(t - \xi) ^ 2\), acting in \(L ^ 2((0,+\infty))\) with Neumann boundary condition
at zero, has been studied in many places. We refer to \cite[Chapter 3]{MR2662319}
for the results we need (we warn the reader that $\xi$ in that reference must be
replaced with $-\xi$ to match our convention). The smallest eigenvalue $\mu(\xi)$
of this operator is an analytic function of the parameter $\xi$, which has a
unique minimum \(\Theta _ 0 \approx 0.59\) attained only for \(\xi = \xi _ 0 =
\sqrt {\Theta _ 0} \approx 0.768\). The corresponding eigenfunction \(\varphi\),
chosen \(L ^ 2\)-normalized and positive, gives rise to the constant \(C _ 1 =
\varphi(0) ^ 2/3\) that is often used in the literature. In fact, we use more
precise estimates of \(\Theta _ 0\) and \(\varphi (0)\) from~\cite{MR2912745}.

The momentum integrals $T _ k \coloneq \int _ 0 ^ {+\infty} t ^ k \varphi(t) ^ 2
\dd t$ can be deduced from the integrals $M _ k \coloneq \int _ 0 ^ {+\infty} (
t-\xi _ 0 ) ^ k \varphi(t) ^ 2 \dd t$ studied in \cite{MR2662319} by expanding
$(t-\xi_0)^k$. The expressions for $M_0$, $M_1$, $M_3$ and $M_4$ can be obtained
from Lemma 3.2.7 and the recurrence formula (3.54) in \cite{MR2662319} (based on
\cite{MR1608449}).

\newpage

\section{Program and tables}\label{sec:details}

This appendix contains the Mathematica code used for the calculations and tables
with the results from the computations.

\begin{verbatim}
terms = 8;

matM[m_] =
  Table[
    1/2 Beta[m + 1, j + k + 1],
    {j, 0, terms},{k, 0, terms}
  ];

matK[m_] =
  Table[
    1/2(2 m^2 Beta[m, j + k + 1]
    - 2m (j + k) If[j + k > 0, Beta[m + 1, j + k], 0]
    + 4j k If[j + k > 1, Beta[m + 2, j + k - 1], 0]),
    {j, 0, terms}, {k, 0, terms}
  ];

matP[m_] =
  Table[
    1/8 Beta[m + 2, j + k + 1],
    {j, 0, terms}, {k, 0, terms}
  ];

smalleig[b_, m_] :=
  Eigensystem[
    {N[matK[m] - m b matM[m] + b^2 matP[m]], N[matM[m]]}, -1
  ];

findInterval[m_] :=
  Module[
    {eigs, vec, pol, btab, bleft, bright,
     rvec, leftval, rightval},
    eigs = smalleig[2*m + 2.25*Sqrt[m], m];
    vec = eigs[[2,1]]/eigs[[2,1,1]];
    pol = Collect[
            Dot[
             {{vec}}.(matK[m] - (m + 5901/10000) b matM[m]
                                 + b^2 matP[m]),
             vec
             ], b][[1,1]];
    btab     = SolveValues[pol==0,b];
    bleft    = Ceiling[btab[[1]]];
    bright   = Floor[btab[[2]]];
    rvec     = Rationalize[vec, 10^(-2)];
    leftval  = Dot[
                {{rvec}}.(matK[m]
                  - (m + 5901/10000) bleft matM[m]
                  + bleft^2 matP[m] ),
                rvec
                ][[1,1]];
    rightval = Dot[
                {{rvec}}.(matK[m]
                  - (m + 5901/10000) bright matM[m]
                  + bright^2 matP[m] ),
                rvec
                ][[1,1]];
    {m, bleft, bright, leftval, rightval, vec,rvec}
  ];

Table[findInterval[m], {m, 1, 56}]//TableForm
\end{verbatim}

\begin{longtblr}[
  caption={Table with overlapping intervals that work for each angular momentum
  \(m\). These are the \texttt {bleft} and \texttt {bright} in the Mathematica
  output.},
  label=tab:intervaltable]
  {colspec={lp{2cm}lp{2cm}lp{2cm}ll}}
  \toprule
  \( m\) & \([b _ {\mathrm {left}},b _ {\mathrm {right}}]\) &
  \( m\) & \([b _ {\mathrm {left}},b _ {\mathrm {right}}]\) &
  \( m\) & \([b _ {\mathrm {left}},b _ {\mathrm {right}}]\) &
  \( m\) & \([b _ {\mathrm {left}},b _ {\mathrm {right}}]\) \\
  \midrule
  \( 1\) & \([ 3, 7]\) &  \(15\) & \([37,42]\) & \(29\) & \([ 68, 73]\) & \(43\) & \([ 98,104]\) \\
  \( 2\) & \([ 6,10]\) &  \(16\) & \([39,44]\) & \(30\) & \([ 70, 76]\) & \(44\) & \([100,106]\) \\
  \( 3\) & \([ 9,13]\) &  \(17\) & \([41,46]\) & \(31\) & \([ 72, 78]\) & \(45\) & \([102,108]\) \\
  \( 4\) & \([11,15]\) &  \(18\) & \([43,49]\) & \(32\) & \([ 74, 80]\) & \(46\) & \([105,110]\) \\
  \( 5\) & \([14,18]\) &  \(19\) & \([46,51]\) & \(33\) & \([ 76, 82]\) & \(47\) & \([107,112]\) \\
  \( 6\) & \([16,20]\) &  \(20\) & \([48,53]\) & \(34\) & \([ 79, 84]\) & \(48\) & \([109,114]\) \\
  \( 7\) & \([18,23]\) &  \(21\) & \([50,55]\) & \(35\) & \([ 81, 86]\) & \(49\) & \([111,117]\) \\
  \( 8\) & \([21,25]\) &  \(22\) & \([52,58]\) & \(36\) & \([ 83, 89]\) & \(50\) & \([113,119]\) \\
  \( 9\) & \([23,28]\) &  \(23\) & \([54,60]\) & \(37\) & \([ 85, 91]\) & \(51\) & \([115,121]\) \\
  \(10\) & \([25,30]\) &  \(24\) & \([57,62]\) & \(38\) & \([ 87, 93]\) & \(52\) & \([118,123]\) \\
  \(11\) & \([28,32]\) &  \(25\) & \([59,64]\) & \(39\) & \([ 89, 95]\) & \(53\) & \([120,125]\) \\
  \(12\) & \([30,35]\) &  \(26\) & \([61,67]\) & \(40\) & \([ 92, 97]\) & \(54\) & \([122,127]\) \\
  \(13\) & \([32,37]\) &  \(27\) & \([63,69]\) & \(41\) & \([ 94,100]\) & \(55\) & \([124,129]\) \\
  \(14\) & \([34,39]\) &  \(28\) & \([65,71]\) & \(42\) & \([ 96,102]\) & \(56\) & \([126,131]\) \\
  \bottomrule
\end{longtblr}

\newpage

\begin{longtblr}[
  caption={The rationalized coefficient vectors used for each \(m\), called
  \texttt {rvec} in the Mathematica code.},
  label=tab:mandvec]{colspec={Q[l,$]Q[l,$]}}
  \toprule
    m & c _ {\mathrm {rat}} \\
  \midrule
  1  & ( 1, \frac {1}{2}, \frac {3}{10}, \frac {1}{12}, \frac {1}{44}, 0, 0, 0, 0)\\
  2  & ( 1, 1, \frac {8}{9}, \frac {4}{9}, \frac {1}{5}, \frac {1}{15}, \frac {1}{49}, 0, 0)  \\
  3  & ( 1, \frac {3}{2}, \frac {19}{11}, \frac {5}{4}, \frac {8}{11}, \frac {5}{16}, \frac {1}{6}, 0, \frac {1}{27})  \\
  4  & ( 1, 2, \frac {17}{6}, \frac {13}{5}, \frac {37}{19}, 1, \frac {6}{7}, -\frac{1}{8}, \frac {3}{10})  \\
  5  & ( 1, \frac {5}{2}, \frac {21}{5}, \frac {14}{3}, \frac {30}{7}, \frac {9}{4}, \frac {10}{3}, -\frac {8}{7}, \frac {23}{14})  \\
  6  & ( 1, 3, \frac {64}{11}, \frac {83}{11}, \frac {25}{3}, \frac {23}{6}, \frac {109}{10}, -\frac {77}{13}, \frac {61}{9})  \\
  7  & ( 1, \frac {7}{2}, \frac {123}{16}, \frac {113}{10}, \frac {286}{19}, \frac {49}{11}, \frac {249}{8}, -\frac {181}{8}, \frac {298}{13})  \\
  8  & ( 1, 4, \frac {49}{5}, 16, \frac {129}{5}, \frac {3}{4}, \frac {637}{8}, -\frac {494}{7}, \frac {535}{8})  \\
  9  & ( 1, \frac {9}{2}, \frac {61}{5}, \frac {65}{3}, \frac {341}{8}, -\frac {185}{13}, \frac {3161}{17}, -\frac {2097}{11}, \frac {1912}{11})  \\
  10 & ( 1, 5, \frac {104}{7}, \frac {254}{9}, \frac {137}{2}, -\frac {1008}{19}, \frac {6433}{16}, -\frac {3228}{7}, \frac {3707}{9})  \\
  11 & ( 1, \frac {11}{2}, \frac {160}{9}, \frac {320}{9}, \frac {1291}{12}, -\frac {1228}{9}, \frac {10583}{13}, -\frac {7153}{7}, \frac {13567}{15})  \\
  12 & ( 1, 6, 21, \frac {478}{11}, \frac {1987}{12}, -\frac {2669}{9}, \frac {24923}{16}, -\frac {18970}{9}, \frac {14907}{8})  \\
  13 & ( 1, \frac {13}{2}, \frac {49}{2}, \frac {413}{8}, \frac {3751}{15}, -\frac {1742}{3}, \frac {19865}{7}, -\frac {12286}{3}, \frac {54539}{15})  \\
  14 & ( 1, 7, \frac {85}{3}, \frac {1014}{17}, \frac {4079}{11}, -\frac {13723}{13}, \frac {54485}{11}, -\frac {83196}{11}, \frac {74471}{11})  \\
  15 & ( 1, \frac {15}{2}, \frac {487}{15}, 67, \frac {7023}{13}, -\frac {23564}{13}, \frac {133205}{16}, -\frac {133709}{10}, \frac {169423}{14})  \\
  16 & ( 1, \frac {367}{46}, \frac {776}{21}, 73, \frac {13152}{17}, -\frac {38645}{13}, \frac {94736}{7}, -\frac {68269}{3}, \frac {146082}{7})  \\
  17 & ( 1, \frac {178}{21}, \frac {460}{11}, \frac {846}{11}, \frac {11987}{11}, -\frac {61003}{13}, \frac {448475}{21}, -\frac {149829}{4}, \frac {383454}{11})  \\
  18 & ( 1, \frac {296}{33}, \frac {1176}{25}, \frac {856}{11}, \frac {19640}{13}, -\frac {50195}{7}, 32814, -\frac {658426}{11}, \frac {792349}{14})  \\
  19 & ( 1, \frac {142}{15}, \frac {158}{3}, \frac {224}{3}, \frac {90773}{44}, -\frac {138506}{13}, 49226, -\frac {2328949}{25}, \frac {1253951}{14})  \\
  20 & ( 1, \frac {229}{23}, \frac {411}{7}, \frac {862}{13}, \frac {69426}{25}, -\frac {46337}{3}, \frac {289057}{4}, -\frac {424759}{3}, \frac {1523611}{11})  \\
  21 & ( 1, \frac {115}{11}, \frac {587}{9}, \frac {669}{13}, \frac {29503}{8}, -\frac {241005}{11}, \frac {520094}{5}, -\frac {1053139}{5}, \frac {2307215}{11})  \\
  22 & ( 1, \frac {186}{17}, \frac {794}{11}, \frac {201}{7}, \frac {38683}{8}, -\frac {152401}{5}, \frac {6323882}{43}, -\frac {3380292}{11}, \frac {7478327}{24})  \\
  23 & ( 1, \frac {80}{7}, \frac {1115}{14}, -\frac {31}{9}, \frac {18793}{3}, -\frac {458357}{11}, \frac {5113712}{25}, -\frac {7047311}{16}, \frac {2274274}{5})  \\
  24 & ( 1, \frac {143}{12}, \frac {701}{8}, -\frac {607}{13}, \frac {88272}{11}, -\frac {448565}{8}, \frac {1961719}{7}, -\frac {8074862}{13}, \frac {5226819}{8})  \\
  25 & ( 1, \frac {149}{12}, \frac {673}{7}, -\frac {1130}{11}, \frac {61031}{6}, -\frac {1264330}{17}, \frac {3029353}{8}, -\frac {7766767}{9}, \frac {19416853}{21})  \\
  26 & ( 1, \frac {116}{9}, \frac {1473}{14}, -\frac {867}{5}, \frac {217028}{17}, -\frac {1265670}{13}, \frac {23742004}{47}, -\frac {8277980}{7}, \frac {11615375}{9})  \\
  27 & ( 1, \frac {107}{8}, \frac {919}{8}, -\frac {2085}{8}, \frac {111121}{7}, -\frac {881456}{7}, \frac {7325226}{11}, -\frac {19200259}{12}, \frac {14228561}{8})  \\
  28 & ( 1, \frac {97}{7}, \frac {876}{7}, -\frac {2565}{7}, \frac {215246}{11}, -\frac {1449586}{9}, \frac {1736539}{2}, -\frac {23533639}{11}, \frac {24220651}{10})  \\
  29 & ( 1, \frac {244}{17}, \frac {4761}{35}, -\frac {3450}{7}, \frac {311010}{13}, -\frac {2243017}{11}, \frac {8964361}{8}, -\frac {19805902}{7}, \frac {32619861}{10})  \\
  30 & ( 1, \frac {163}{11}, \frac {1033}{7}, -\frac {14125}{22}, \frac {145127}{5}, -\frac {2557077}{10}, \frac {12891229}{9}, -\frac {37038129}{10}, \frac {13043512}{3})  \\
  31 & ( 1, \frac {153}{10}, \frac {1438}{9}, -\frac {4081}{5}, \frac {454492}{13}, -\frac {2860535}{9}, \frac {12702547}{7}, -\frac {33615781}{7}, \frac {63130052}{11})  \\
  32 & ( 1, \frac {142}{9}, \frac {518}{3}, -\frac {11193}{11}, \frac {501893}{12}, -\frac {4701821}{12}, \frac {25077887}{11}, -\frac {98734523}{16}, \frac {120107739}{16})  \\
  33 & ( 1, \frac {65}{4}, \frac {2049}{11}, -\frac {13732}{11}, \frac {546868}{11}, -\frac {7669029}{16}, \frac {22734343}{8}, -\frac {62905815}{8}, \frac {136288637}{14})  \\
  34 & ( 1, \frac {117}{7}, \frac {1003}{5}, -\frac {22664}{15}, \frac {646121}{11}, -\frac {11060612}{19}, \frac {35162523}{10}, -\frac {89466803}{9}, \frac {100182269}{8})  \\
  35 & ( 1, \frac {86}{5}, \frac {3666}{17}, -\frac {14461}{8}, \frac {138007}{2}, -\frac {9129316}{13}, \frac {47527098}{11}, -\frac {62369911}{5}, \frac {175852778}{11})  \\
  36 & ( 1, \frac {53}{3}, \frac {3009}{13}, -\frac {23548}{11}, \frac {403131}{5}, -\frac {5892557}{7}, \frac {26371834}{5}, -\frac {202062129}{13}, \frac {182357222}{9})  \\
  37 & ( 1, \frac {127}{7}, \frac {5705}{23}, -\frac {20101}{8}, \frac {1312173}{14}, -\frac {7021314}{7}, \frac {63988747}{10}, -\frac {384793359}{20}, \frac {1198776030}{47})  \\
  38 & ( 1, \frac {93}{5}, \frac {1327}{5}, -\frac {32182}{11}, \frac {2168601}{20}, -\frac {10696210}{9}, 7717818, -\frac {473324147}{20}, \frac {350899896}{11})  \\
  39 & ( 1, \frac {286}{15}, \frac {3119}{11}, -\frac {23675}{7}, \frac {1498397}{12}, -\frac {18209036}{13}, \frac {101828840}{11}, -\frac {405141307}{14}, \frac {198264446}{5})  \\
  40 & ( 1, \frac {254}{13}, \frac {5747}{19}, -\frac {73805}{19}, \frac {715852}{5}, -1642536, \frac {143589502}{13}, -35186816, \frac {539037065}{11})  \\
  41 & ( 1, 20, \frac {2900}{9}, -\frac {84264}{19}, \frac {980887}{6}, -\frac {11501881}{6}, \frac {39340096}{3}, -\frac {297888954}{7}, \frac {542013226}{9})  \\
  42 & ( 1, \frac {184}{9}, \frac {1714}{5}, -\frac {40287}{8}, \frac {1301596}{7}, -\frac {55680026}{25}, \frac {123959643}{8}, -\frac {204824615}{4}, \frac {1766954375}{24})  \\
  43 & ( 1, \frac {230}{11}, \frac {1821}{5}, -\frac {51206}{9}, \frac {421403}{2}, -\frac {28342151}{11}, \frac {164040335}{9}, -\frac {429221979}{7}, \frac {985067489}{11})  \\
  44 & ( 1, \frac {171}{8}, \frac {4637}{12}, -\frac {83177}{13}, \frac {1665375}{7}, -\frac {41560475}{14}, \frac {277526491}{13}, -\frac {1169394213}{16}, 108403952)  \\
  45 & ( 1, \frac {131}{6}, \frac {6142}{15}, -\frac {42985}{6}, \frac {1070903}{4}, -\frac {27256667}{8}, \frac {124510652}{5}, -\frac {954066651}{11}, \frac {3265601563}{25})  \\
  46 & ( 1, \frac {156}{7}, \frac {3467}{8}, -\frac {103864}{13}, \frac {1501531}{5}, -\frac {35063396}{9}, \frac {173607013}{6}, -\frac {717470237}{7}, 156709469)  \\
  47 & ( 1, \frac {91}{4}, \frac {3665}{8}, -\frac {168654}{19}, \frac {2350708}{7}, -\frac {310750439}{70}, \frac {535915183}{16}, -\frac {2050834167}{17}, \frac {561648773}{3})  \\
  48 & ( 1, \frac {116}{5}, \frac {3386}{7}, -\frac {108100}{11}, \frac {4118623}{11}, -\frac {40331997}{8}, \frac {424992302}{11}, -\frac {1131564945}{8}, \frac {1336585319}{6})  \\
  49 & ( 1, \frac {331}{14}, \frac {6632}{13}, -\frac {86751}{8}, \frac {5828075}{14}, -\frac {256819411}{45}, \frac {888280661}{20}, -\frac {1817580518}{11}, \frac {2112338615}{8})  \\
  50 & ( 1, \frac {217}{9}, \frac {4837}{9}, -\frac {83499}{7}, \frac {3231210}{7}, -\frac {186784214}{29}, \frac {661573630}{13}, -\frac {7886187967}{41}, \frac {4677217849}{15})  \\
  51 & ( 1, \frac {221}{9}, \frac {2828}{5}, -\frac {156995}{12}, \frac {2552641}{5}, -\frac {115962655}{16}, \frac {930061685}{16}, -\frac {2677816973}{12}, \frac {3302332628}{9})  \\
  52 & ( 1, 25, \frac {2973}{5}, -\frac {114475}{8}, \frac {6195764}{11}, -\frac {81327563}{10}, \frac {728181742}{11}, -\frac {4128855659}{16}, \frac {1721235577}{4})  \\
  53 & ( 1, \frac {229}{9}, \frac {5620}{9}, -\frac {62439}{4}, \frac {7439443}{12}, -\frac {136522249}{15}, \frac {751715847}{10}, -\frac {1784929117}{6}, \frac {1508971868}{3})  \\
  54 & ( 1, \frac {259}{10}, \frac {4586}{7}, -\frac {271775}{16}, \frac {20424629}{30}, -\frac {274304825}{27}, \frac {681006877}{8}, -\frac {4445044522}{13}, \frac {5274895759}{9})  \\
  55 & ( 1, \frac {369}{14}, \frac {6867}{10}, -\frac {239718}{13}, \frac {10444589}{14}, -\frac {79186992}{7}, \frac {1057573046}{11}, -\frac {2351257957}{6}, \frac {4766112306}{7})  \\
  56 & ( 1, \frac {134}{5}, \frac {6472}{9}, -\frac {139813}{7}, \frac {10605507}{13}, -\frac {25132953}{2}, \frac {324928729}{3}, -\frac {5822498153}{13}, \frac {21293986697}{27})  \\
  \bottomrule
\end{longtblr}

\newpage

\begin{longtblr}[
  caption={The exact values of \(\Quadraticform[u _ {\mathrm {rat}}] - \Theta _ *
  b \norm {u _ {\mathrm {rat}}} ^2\) at \(b _ {\mathrm {left}}\) and \(b _
  {\mathrm {right}}\), for each \(m\). Note that each entry in the last two
  columns is negative. These are the \texttt {leftval} and \texttt {rightval}
  in the Mathematica code.},
  label=tab:results]{colspec={Q[l,$]Q[l,$]Q[l,$]}}
  \toprule
  m & \text {value at }b _ {\mathrm {left}} & \text {value at }b _ {\mathrm {right}} \\
  \midrule
  1 & -\frac{6193251389}{8944320000000} & -\frac{5982585742501}{34499520000000} \\
  2 & -\frac{33022175272883}{417161745000000} & -\frac{76709166158227}{584026443000000} \\
  3 & -\frac{336498870762680359}{2021848509542400000} & -\frac{1503630262025024519}{26284030624051200000} \\
  4 & -\frac{1601408321447213}{22813670880000000} & -\frac{22400280127327309}{130037924016000000} \\
  5 & -\frac{10368109835804321}{58663725120000000} & -\frac{345872985922211}{4147940160000000} \\
  6 & -\frac{964029034387889419}{9238102849725750000} & -\frac{33158383315855019}{191960578695600000} \\
  7 & -\frac{27310328006094742153}{1024139766633984000000} & -\frac{47742143270790384797}{682759844422656000000} \\
  8 & -\frac{3317377294217839}{21416915520000000} & -\frac{99142875671057}{666304038400000} \\
  9 & -\frac{72548871846993693166187}{759413112482344320000000} & -\frac{6075756196320077791883}{189853278120586080000000} \\
  10 & -\frac{191105242656754164667}{5891779067384770560000} & -\frac{18940072667147079689}{175350567481689600000} \\
  11 & -\frac{2915378356030456753}{17861572563120000000} & -\frac{104391638477595003869}{619945414378290000000} \\
  12 & -\frac{10175680571251334171}{88542668819404800000} & -\frac{26782574710875620447}{531256012916428800000} \\
  13 & -\frac{155548020980941982479}{2449306416387600000000} & -\frac{115550198332966313522399}{1018911469217241600000000} \\
  14 & -\frac{350031030295502019181}{36784072635863553000000} & -\frac{162991510786164680945771}{984681329021578188000000} \\
  15 & -\frac{4870353240661634747929603}{32799245390040729600000000} & -\frac{37060448996288648432321}{863138036580019200000000} \\
  16 & -\frac{1901762373426052714199969}{17884172915076622354200000} & -\frac{4738345936456410148037}{47438124443174064600000} \\
  17 & -\frac{25940704490081011102996981}{416714412680467469568000000} & -\frac{2803659139220165189034713}{18941564212748521344000000} \\
  18 & -\frac{1945706277421672804879081}{120577087002450078000000000} & -\frac{6569662788649157360439649}{361731261007350234000000000} \\
  19 & -\frac{750146398117776009424213901}{4876781777953345260000000000} & -\frac{336227611229919228725799503}{4696160230621739880000000000} \\
  20 & -\frac{81761271972425561956184760643}{692624932013823860551500000000} & -\frac{10041621097209236972738666618063}{84961991660362393560984000000000} \\
  21 & -\frac{11037178589519902746911}{136994830257032748000000} & -\frac{3186385007735868534392539}{20113331896827989820000000} \\
  22 & -\frac{117449737156898311081239013}{2849500551206303074080000000} & -\frac{7792747478669148030470589419}{256047978101252090513760000000} \\
  23 & -\frac{376926945403786681666742159}{813166663283083549440000000000} & -\frac{18728821133255831423620010053}{243949998984925064832000000000} \\
  24 & -\frac{78678150415237054257689489}{555253196342792403456000000} & -\frac{78872622544826153191086293}{670366663877273755392000000} \\
  25 & -\frac{222654373707321552622481}{2022770522436254791200000} & -\frac{4360344621994138090484719813}{28420937225490597943755600000} \\
  26 & -\frac{119934820763695455926842865059}{1568641272927952145206284000000} & -\frac{136487991922294627118650789}{5820561309565685139912000000} \\
  27 & -\frac{6500834855915161129978147}{155405184538544856115200000} & -\frac{1449611123956502136782023}{22200740648363550873600000} \\
  28 & -\frac{6940534355834380582054729}{1159935331711975428456000000} & -\frac{7747635234229781958774775259}{75395796561278402849640000000} \\
  29 & -\frac{436309597754176623899142269429}{3015016878365553227886720000000} & -\frac{12562272842071337120716526784241}{92460517603210298988526080000000} \\
  30 & -\frac{2972617077052848688595617}{25576805343871848528000000} & -\frac{942784271167061333486647}{2143740132111364146360000000} \\
  31 & -\frac{57835527175019744931493327669}{666149957826734542641480000000} & -\frac{5512642530059209694870030521}{140242096384575693187680000000} \\
  32 & -\frac{1977431414199117461755939}{35159554914880325529600000} & -\frac{2682119272992269173967027}{36038543787752333667840000} \\
  33 & -\frac{1217238906598191846894881}{49547941795731322944000000} & -\frac{10506581398837470506966059}{99095883591462645888000000} \\
  34 & -\frac{13828154197780707571362405397}{88934961085595846854132500000} & -\frac{27309114248048830551653819881}{203279911052790507095160000000} \\
  35 & -\frac{962055064061900682722617866497}{7363866108450429542764800000000} & -\frac{37751057336381914233284945849}{237211373241400391433600000000} \\
  36 & -\frac{5529588727463486404395627144247}{52958921647225396139997660000000} & -\frac{1689598609863850108058206572301}{52958921647225396139997660000000} \\
  37 & -\frac{507042742155628158535116557081}{6556327310920551853328640000000} & -\frac{6727867072466313762472618063}{108909091543530761683200000000} \\
  38 & -\frac{1485994428194626897507591}{30323530861849625400000000} & -\frac{147437546144919132785681}{1658318094007401389062500} \\
  39 & -\frac{13030862939788118739250693649}{652454567148384028828800000000} & -\frac{5428424962791714486079094287}{48075599684617770545280000000} \\
  40 & -\frac{292624718857192541745243599}{1998465750070075088302500000} & -\frac{229641318843426236808298697}{1712970642917207218545000000} \\
  41 & -\frac{3816483376933398943225366051}{31079435007889757790100800000} & -\frac{867846409313144725209533}{621588700157795155802016000} \\
  42 & -\frac{2781535999263050210170574029}{28398606549606869325750000000} & -\frac{49598367656105787491422242097}{1817510819174839636848000000000} \\
  43 & -\frac{10385014473906096147931466089}{143255133135934293113640000000} & -\frac{5455034794715697883834249}{109045882117392350130000000} \\
  44 & -\frac{66561697062280505897479982221}{1440839169193367010987279360000} & -\frac{5051359871350693964197267256663}{72041958459668350549363968000000} \\
  45 & -\frac{25217985292106568834428336177}{1329116087456479786464000000000} & -\frac{1496087272489576888907870111}{17039949839185638288000000000} \\
  46 & -\frac{2041211787766542997265951623561}{15101963341700663546838720000000} & -\frac{13086528371405089789892936033}{126907254972274483586880000000} \\
  47 & -\frac{29220932036087572035969234282283}{259839876885004951277783040000000} & -\frac{1738894916893820193935712101}{14995376089854856375680000000} \\
  48 & -\frac{634450535525681137718102366399}{7157308573822072278412800000000} & -\frac{914399806218828806257975037}{7215028804256121248400000000} \\
  49 & -\frac{3456391813829523739132767444931}{54128374384600068250553856000000} & -\frac{1024414205156603366779476782599}{378898620692200477753876992000000} \\
  50 & -\frac{564571920337878645123969801154549}{14632237132419098328526243854000000} & -\frac{33665602814369844722560596939101}{2034428692208002976158729092000000} \\
  51 & -\frac{270601306148913019907435999}{22184328239801252388864000000} & -\frac{72675389850412492184509721879}{2551197747577144024719360000000} \\
  52 & -\frac{162288699932006763456678153799}{1429143182670529754588160000000} & -\frac{27325543940257282510795731257}{714571591335264877294080000000} \\
  53 & -\frac{145474238907266196249629}{1604587151452736162250000} & -\frac{262561231656431052423787}{5721499328608613515680000} \\
  54 & -\frac{426640334586937465827392799744199}{6367836510363945593537713152000000} & -\frac{11095151346146113583315694748927}{215858864758099850628397056000000} \\
  55 & -\frac{19704901773149803184653426212313}{463680983367995636044161360000000} & -\frac{305781044177254996599356745902101}{5564171800415947632529936320000000} \\
  56 & -\frac{326354424328040632707646485015239}{19195900016175242241233664312000000} & -\frac{448703701992931931259989204401289}{7904194124307452687566802952000000} \\
  \bottomrule
\end{longtblr}

\newpage

\bibliographystyle{plain}
\bibliography{LS-diskeig}{}

\end{document}